\documentclass[11pt]{article}
\usepackage{enumerate}
\usepackage{amsthm,amsmath,amssymb}
\usepackage{graphicx}
\usepackage{lineno}
\usepackage[colorlinks=true,citecolor=black,linkcolor=black,urlcolor=blue]{hyperref}
\usepackage[english]{babel}
\usepackage{amsfonts}
\usepackage{epsfig, subfigure}
\usepackage{amscd,latexsym}
\usepackage{url}
\usepackage{color}
\usepackage{authblk}
\usepackage{subfigure}

\theoremstyle{plain}
\newtheorem{theorem}{Theorem}
\newtheorem{lemma}[theorem]{Lemma}
\newtheorem{cor}[theorem]{Corollary}
\newtheorem{prop}[theorem]{Proposition}
\newtheorem{conj}[theorem]{Conjecture}

{\itshape}{\rmfamily}

\newtheorem{defi}[theorem]{Definition}

\begin{document}

\title{Some Families of Graphs with Small Power Domination Number}

\author[1]{Najibeh Shahbaznejad\thanks{najibeh.shahbaznejad@uma.ac.ir}}
\author[1]{Adel P. Kazemi\thanks{a.kazemi@uma.ac.ir}}
\author[2]{Ignacio M. Pelayo\thanks{ignacio.m.pelayo@upc.edu}}

\affil[1]{Department of Mathematics, University of Mohaghegh Ardabili, Iran}
\affil[2]{Departament de Matem\`atiques, Universitat Polit\`ecnica de Catalunya, Spain}

%\date{}

\maketitle

%\linenumbers

\begin{abstract}

Let $ G $ be a graph with the vertex set $ V(G) $ and $ S $ be a subset of $ V(G) $.
Let $cl(S)$ be the set of vertices built from  $S$,  by iteratively applying  the following propagation rule:
if a vertex and all of its neighbors except one of them are in $ cl(S) $, then the exceptional neighbor is also in $ cl(S) $.
A set  $S$ is called a zero forcing set of $G$ if $cl(S)=V(G)$.
The zero forcing number $Z(G)$ of $G$ is the minimum cardinality of a zero forcing set. 
Let $cl(N[S])$ be the set of vertices built from the closed neighborhood $N[S]$ of  $S$,  by iteratively applying  the previous propagation  rule.
A set  $S$ is called a power dominating set of $G$ if $cl(N[S])=V(G)$.
The power domination number $\gamma_p (G)$ of $G$ is the minimum cardinality of a power dominating set. 
In this paper, we present some families of graphs that their power domination number is 1 or 2.

\vspace{+.1cm}\noindent \textbf{Keywords:} domination, power domination, maximum degree, mycieleskian graphs, central graphs, middle graphs.

\vspace{+.1cm}\noindent \textbf{AMS subject classification:} 05C38, 05C76, 05C90.
\end{abstract}
\vspace{0.5cm}

%%%%%%%%%%%%%%%%%%%%%%%%%Introduction%%%%%%%%%%%%%%%%%%%%%%%%%%%%%%%%%%%
%      introduccion
%%%%%%%%%%%%%%%%%%%%%%%%%Introduction%%%%%%%%%%%%%%%%%%%%%%%%%%%%%%%%%%%

%%%%%%%%%%%%%%%%%%%%%%%%%%%%%%%%%%%%%%%%%%%%%%%%%%%%%%%%%%%%%%%%%%%%%%%%%%%%%%%%%%%%%%%%%%%%%%%%%%
%%%%%%%%%%%%%%%%%%%%%%%%%%%%%%%%%%%%%%%%%%%%%%%%%%%%%%%%%%%%%%%%%%%%%%%%%%%%%%%%%%%%%%%%%%%%%%%%%%
%%%%%%%%%%%%%%%%%%%%%%%%%%%%%%%%%%%%%%%%%%%%%%%%%%%%%%%%%%%%%%%%%%%%%%%%%%%%%%%%%%%%%%%%%%%%%%%%%%
%%%%%%%%%%%%%%%%%%%%%%%%%%%%%%%%%%%%%%%%%%%%%%%%%%%%%%%%%%%%%%%%%%%%%%%%%%%%%%%%%%%%%%%%%%%%%%%%%%
\section{Introduction}\label{sec1:intro}

%%%%%%%%%%%%%%%%%%%%%%%%
%%%%%%%%%%%%%%%%%%%%%%%%
%%%%%

This paper is devoted to the study of both the power domination number of connected graphs introduced in \cite{hahehehe02}. 

The notion of power domination in graphs is a dynamic version of domination where a set of vertices  (power) dominates larger and larger portions of a graph and eventually dominates the whole graph.
The introduction of this parameter  was mainly  inspired by a problem in the electric power system industry \cite{bamiboad93}. 
Electric power networks must be continuously monitored. 
One usual and efficient way of accomplish this monitoring, consist in placing phase measurement units (PMUs), called PMUs,  at selected network locations. 

Due to the high cost of the PMUs, their number must be minimized,  while maintaining the ability to monitor (i.e. to observe)  the entire network. 
The \emph{power domination problem} consists thus  of finding the minimum number of PMUs needed to monitor a given electric power system.
In other words,  a power dominating set of a graph is a set of vertices that observes every vertex in the graph, following the set of rules for power system monitoring described in \cite{hahehehe02}.

Since it was formally introduced  in \cite{hahehehe02}, the power domination number has generated considerable interes; see, for example, \cite{bfffhv18,domoklsp08,dovavi16,fehokeyo17,gunira08,zhkach06}.

The defnition of the power domination number leads naturally to the introduction and  study of the zero forcing number.
As a matter of fact, the zero forcing number of a connected graph $G$ was
introduced in \cite{AIM08} as a tight  upper bound for the maximum nullity of the set of all  real symmetric matrices whose pattern of off-diagonal entries coincides with off-diagonal entries of the adjacency matrix of $G$,
and independently by mathematical physicists studying control of quantum systems \cite{bugi07}.
Since then, this parameter has been extensively investigated; see, for example,  \cite{dakast18,erkayi17,geperaso16,gera18,huchye10,kakasu19}.

In this paper, we present a variety of graph families such that all theirs members have power dominating sets of cardinality at most 2.
%%%%

%\vspace{1.1cm}
%%%%%%%%%%%%%%%%%%%%%%%%%%%%%%%%%%%%%%%%%%%%%%%%%%%%%%%%%%%%%%%%%%%%%%%%%%%%%%%%%%%%%%%%%%%%%%%%%%
%%%%%%%%%%%%%%%%%%%%%%%%%%%%%%%%%%%%%%%%%%%%%%%%%%%%%%%%%%%%%%%%%%%%%%%%%%%%%%%%%%%%%%%%%%%%%%%%%%
\subsection{Basic terminology}

{\small All the graphs considered are undirected, simple, finite and (unless otherwise stated) connected. Let $G = (V (G),E(G))$ be a graph which $V (G)$ and $E(G)$ are the vertex set and the edge set of $G$, respectively.
Let $v$ be a vertex of a graph $G$.
The \emph{open neighborhood} of $v$ is $\displaystyle N_G(v)=\{w \in V(G) :vw \in E\}$, and the \emph{closed neighborhood} of $v$ is $N_G[v]=N_G(v)\cup \{v\}$ (we will write $N(v)$ and $N[v]$ if the graph $G$ is clear from the context).
The \emph{degree} of $v$ is $\deg(v)=|N(v)|$.
The minimum degree  (resp. maximum degree) of $G$ is $\delta(G)=\min\{\deg(u):u \in V(G)\}$ (resp. $\Delta(G)=\max\{\deg(u):u \in V(G)\}$).
If $\deg(v)=1$, then $v$ is said to be a  \emph{leaf} of $G$.

The distance between vertices $v,w\in V(G)$ is denoted by $d_G(v,w)$, or $d(v,w)$ if the graph $G$ is clear from the context.
The diameter of $G$ is ${\rm diam}(G) = \max\{d(v,w) : v,w \in V(G)\}$.
Let $W\subseteq V(G)$.
The  \emph{open neighborhood} of $W$ is $N(W)=\cup_{v\in W} N(v)$ and the  \emph{closed neighborhood} of $W$ is $N[W]=\cup_{v\in W} N[v]$.
Let $u,v \in V(G)$ be  a pair of vertices such that  $d(u,w)=d(v,w)$ for all $w\in V(G)\setminus\{u,v\}$, i.e., either $N(u)=N(v)$ or $N[u]=N[v]$. In both cases, $u$ and $v$ are said to be \emph{twins}.

Let $H$ and $G$ be a pair of graphs.
The graph $ H $ is called a \emph{subgraph} of $ G $ if it can be obtained from $ G $ by removing some of the edges and vertices. Moreover, a subgraph $ H $ of $ G $ is called \emph{induced} if it can be obtained from $ G $ by removing some vertices. 
An induced subgraph of $ G $ made by $ W $, a subset of $ V(G) $, is denoted by $ G[W] $, and its edge set is $ \{vw \in E(G) : v \in W,w \in W\} $.
The graph $H$ is a  \emph{minor} of $G$ if it can be obtained from $G$ by removing some vertices and by removing and contracting edges.
 
Let $ D $ be a subset of $ V(G) $. $ D $ is called a \emph{dominating set} whenever $N[D]=V(G)$. Also, the \emph{domination number} $\gamma(G)$ is the cardinality of smallest dominating set for $ G $.

Let $K_n$, $K_{h,n-h}$, $ K_{1,n-1}$, $P_n$, $W_n$ and $C_n$ denote complete graph, complete bipartite graph, star, path, wheel and cycle, respectively, which the order of each is n.
For undefined terminology and notation,  we refer the reader to \cite{chlezh11}.

The remainder of this paper is organized into five more sections as follows. 
Continuing this section is devoted to introducing the zero forcing sets, the zero forzing number $Z(G)$ of a connected graph $G$, power dominating sets and the power domination number $\gamma_p(G)$ of a connected graph $G$ are first introduced.
In Section 2, we present a brief description of known and new results. 
In Section 3, some contributions involving graphs with high maximum degree are presented.
Finally, in Section 4 to 6, some results are presented for mycieleskian graphs, central graphs and middle graphs.

The concept of \emph{zero forcing} can be described via the following coloring game on the vertices of a given graph $G=(V,E)$. 
Let $U$ be a proper subset of $V$.
The elements of $U$ are colored black, meanwhile the vertices of $W=V\setminus U$ are colored white.
The color change rule is: 

\begin{center}
{\bf \small If $u \in U$ and exactly one neighbor $w$ of $u$ is white, then change the color of $w$ to black.}
\end{center}

In such a case, we  denote this by $u \rightarrow w$, and we say, equivalentely, that   $u$ forces $w$, that $u$ is a forcing vertex of $w$ and also that $u \rightarrow w$ is a force. 
The \emph{closure} of $U$, denoted $cl(U)$, is the set of black vertices obtained after the color change rule is applied until no new
vertex can be forced; it can be shown that $cl(U)$ is uniquely determined by $U$ (see \cite{AIM08}).

\begin{defi}[\cite{AIM08}]
A set $U \subseteq V(G)$ is called a \emph{zero forcing set} of $G$ if $cl(U)=V(G)$. 
\end{defi}

A \emph{minimum zero forcing set}, a \emph{ZF-set} for short,  is a zero forcing set of minimum cardinality. 
The \emph{zero forcing number} of $G$ , denoted by $Z(G)$, is the cardinality of a ZF-set.

A \emph{chronological list of forces} ${\cal F}_U$ associated with a  set $U$ is a sequence
of forces applied to obtain $cl(U)$ in the order they are applied. 
A \emph{forcing chain} for the chronological list of forces ${\cal F}_U$ is a maximal sequence of vertices $(v_1, . . . , v_k )$ such that
the force $v_i \rightarrow v_{i+1}$   is in ${\cal F}_U$ for $1 \le i  \le k-1$.
Each forcing chain induces a distinct path in $G$, one of whose endpoints is
in $U$; the other is called a terminal.
Notice that a zero forcing chain can consist of a single vertex $(v_1)$, and this happens   if $v_1 \in U$ and $v_1$ does not perform a force.
Observe also that any two forcing chains are disjoint.

%%%%%%%%%%%%%%%%%%%%%
\begin{prop}[\cite{erkayi17}]\label{Z1}
Let $G$ be a graph of order $n$.
Then,  $Z(G)=1$ if and only if $G$ is the  path $P_n$.
\end{prop}
%%%%%%%%%%%%%%%%%%%%%

A graph is outerplanar if it has a crossing-free embedding in
the plane such that all vertices are on the same face.
The \emph{path cover number} $P(G)$ of a graph $G$ is the smallest positive integer $k$ such
that there are $k$ vertex-disjoint induced paths $P_1, \ldots, P_k$ in $G$  that cover all
the vertices of $G$, i.e., $\displaystyle V(G) = \bigcup _{i=1}^k V(P_i)$.

%%%%%%%%%%%%%%%%%%%%%
\begin{prop}[\cite{bbfhhsh10}] 
For any graph  $G$, $P(G) \le Z(G)$.
\end{prop}
%%%%%%%%%%%%%%%%%%%%%

%%%%%%%%%%%%%%%%%%%%%
\begin{theorem}[\cite{ddr12}] \label{mainzf}
Let $G$ be a graph of order $n\ge5$.
Then,  $Z(G)=2$ if and only if $G$ is an outerplanar graph with $P(G)=2$.
\end{theorem}

Zero forcing is closely related to power domination, because power
domination can be described as a domination step followed by the zero forcing process or, equivalentely,  zero forcing can be described as power domination without the domination step.
In other words, the power domination process on a graph $G$ can be described
as choosing a set $S \subset V (G)$ and applying the zero forcing process to the closed neighbourhood $N[S]$ of $S$. 
The set $S$ is thus a power dominating set of $G$ if and only if $N[S]$ is a zero forcing set for $G$.

%%%%%%%%%%%%%%%%%%%%%
\begin{defi}[\cite{hahehehe02}]{\rm
Let $ G $ be a graph and $ S \subseteq V(G) $. The set $S$ is called a \emph{power dominating set} of $G$ if $cl(N[S])=V(G)$. }
\end{defi}
%%%%%%%%%%%%%%%%%%%%%

\noindent A \emph{minimum power dominating set}, a \emph{PD-set} for short,  is a power dominating set of minimum cardinality. 
The \emph{power dominating number} of $G$ , denoted by $\gamma_p (G)$, is the cardinality of a PD-set.

\begin{defi}[\cite{vavi16}]
If $ G $ is a graph and $ S_0=S \subseteq V(G) $, then the sets
$ S_i ( i > 0) $ of vertices monitored by $S_0$ at step $ i $ are as follows:

$ S_1=N_{G}[S_0] $ (domination step), and

$ S_{i+1}= \bigcup \{ N_{G}[v]:  v \in S_i ~such~ that \vert N_{G}[v]  \setminus S_i \vert \leq 1 $  (propagation steps).
\end{defi}
%%%%%%%%%%%%%%%%%%%%%%%%%%%%%%%%%%%%%%%%%%%%%%%%%%%%%%%%%%%%%%%%%%%%%%%%%%%%%%%%%%%%%%%%%%%%%%%%%%
%%%%%%%%%%%%%%%%%%%%%%%%%%%%%%%%%%%%%%%%%%%%%%%%%%%%%%%%%%%%%%%%%%%%%%%%%%%%%%%%%%%%%%%%%%%%%%%%%%
\section{Basic Results}

\noindent As a straight consequence of these definitions,  it is derived both  that $\gamma_p(G) \le Z(G)$ and $\gamma_p(G) \le \gamma(G)$.
Moreover, this pair of inequalities along with Theorem \ref{mainzf}, allow us to derive the following results.

%%%%%%%%%%%%%%%%%%%%%%%
\begin{cor}
Let $G$ be a graph of order $n$.
\begin{itemize}
\item If $G$ is outerplanar and $P(G)=2$, then $\gamma_p(G) \le 2$.
\item   $\Delta(G)=n-1$ if and only if $\gamma_p(G) = \gamma(G)=1$.
\end{itemize}
\end{cor}
%%%%%%%%%%%%%%%%%%%%%%%

We end this section by presenting a first list of new and known results involving this parameter along with a Table containing  information of some basic graph families.

%%%%%%%%%%%%%%%%%%%%%%%%%%%%%%%%%%%%
\begin{prop}
If $G$ is a connected graph of order al most 5, then  $\gamma_p(G)=1$.
Moreover,
\begin{itemize}
	\item The smallest connected graph $G$  such that $\gamma_p(G)=2$ is the H-graph (see Figure \ref{HW8} (a)).
	\item Three of the smallest connected graph $G$  with no twin vertices such that $\gamma_p(G)=2$ are in the Figure \ref{HW8} (b), (c) and (d).
	\end{itemize}
\end{prop}
%%%%%%%%%%%%%%%%%%%%%%%%%%%%%%%%%%%%

%%%%%%%%%%%%%%%%%%%%%%%%%%%%%%%%%%%%
\begin{figure}[!h]
	\centerline{\includegraphics[height=4cm]{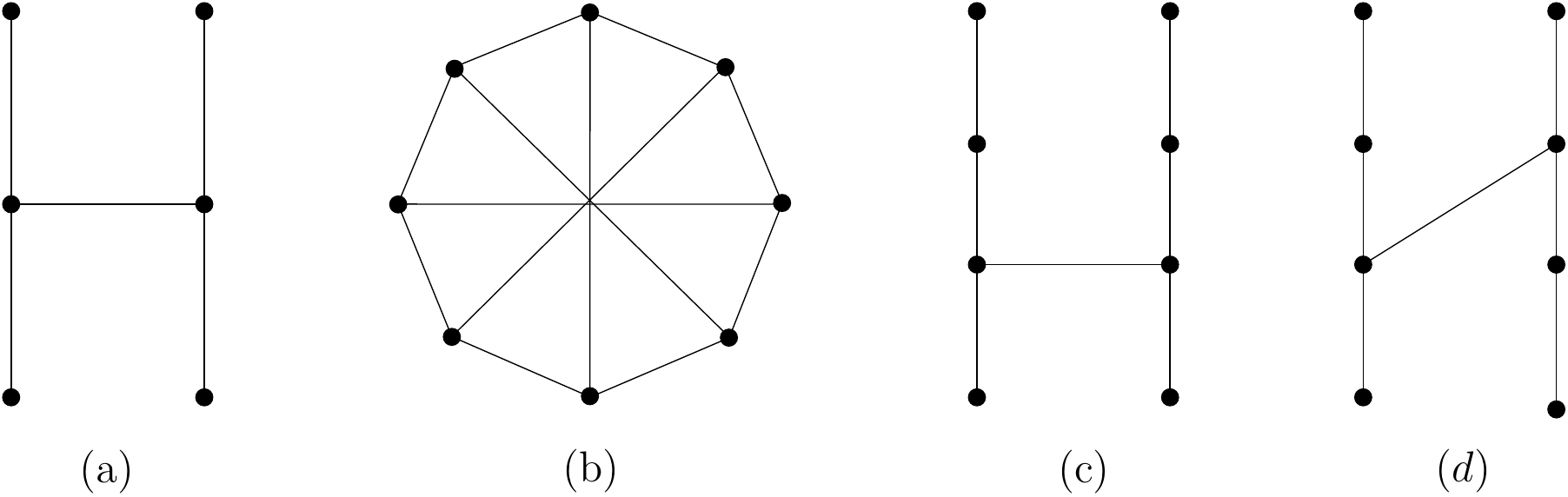}}
	\caption{Some small graphs}
	\label{HW8}
\end{figure}
%%%%%%%%%%%%%%%%%%%%%%%%%%%%%%%%%%%%

%%%%%%%%%%%%%%%%%%%%%%%%%%%%%%%%%%%%
\begin{table}[ht]\label{tb1}
\begin{center}

\begin{tabular}{|c|cccccccc|} \hline
$G$   &   $P_n$     &  $C_n$   & $K_n$ &   $K_{1,n}$ &   $K_{2,n}$ &   $K_{h,n-h}$ &  $W_n$  & \\ \hline \hline
$\gamma_p (G)$ & 1   &    1   &   1  &   1  &   1   &  2   & 1  &       \\ \hline
$\gamma(G)$ & {\scriptsize$\lfloor\frac{n+2}{3}\rfloor$} & {\scriptsize$\lfloor\frac{n+2}{3}\rfloor$} &    1   & $1$   & 2 & 2 & 1 &    \\  \hline
$Z(G)$      &  1 & 2 &    $n-1$   & $n-2$   & $n-2$ & $n-2$ & 3 &    \\  \hline

\end{tabular}
 \caption{Power domination, domination  and zero forcing numbers of some basic graph families.}
 \label{lio}
 \end{center}
\end{table}
%%%%%%%%%%%%%%%%%%%%%%%%%%%%%%%%%%%%%%

%%%%%%%%%%%%%%%%%%%%%
\begin{prop}
	Let $G = K_{r_{1},\cdots,r_k}$ be the complete $k$-partite graph  with  $2\le k$ and  $1 \le r_{1} \leq r_{2} \leq \cdots \leq r_{k} $ and $V(G)=\cup_{i=1}^kV_i$.
	Let $G_e$ the graph obtained from $G$ by deleting an edge $ e=vw \in E(G)$.
	Then
	\begin{enumerate}[{\rm (1)}] 
	
	\item If $ r_{1} \leq 2 $, then $ \gamma_p(G)=1 $. 
	
    \item If $ r_{1} \geq 3 $, then $ \gamma_p(G)=2 $.
    
	\item If $r_1\le 2$, then $ \gamma_p(G_e)=1$.
	
	\item If $r_1=3$, then 
	$\gamma_p(G_e)=\left\{\begin{array}{lll}
    1, & {\rm if} & \{v,w\}\cap V_1\neq \emptyset \\
      
    2, &  & otherwise. \\
\end{array}\right.$

	\item If $4\le r_1$, then $ \gamma_p (G_e)=2$.
	
	\end{enumerate}
\end{prop}
\begin{proof}
\begin{enumerate}[{\rm (1)}] 
	
	\item Take $v_1\in V_1$. 
	Notice that $N[v_1]=V(G)\setminus [V_1-v_1]$.
	If $r_1=1$, then $\{v_1\}$ is a dominating set of $G$, i.e., $\gamma_p(G)=1$.
	Suppose that $r_1=2$ and $V_1=\{v_1,v_1'\}$.
	Then, for any vertex $u\not\in V_1$, $u \rightarrow v_1'$, which means that $\gamma_p(G)=1$, as $N[v_1]=V(G)\setminus \{v_1'\}$.
	
	\item For every $u\in V_i$, $N[u]=V(G)\setminus [V_i-u]$.
	Thus, $\gamma_p(G)\ge2$.
	Take $S=\{v_1,v_2\}$, where $v_1\in V_1$ and $v_2\in V_2$.
	Notice that $N[S]=V(G)$.
	Hence, $\gamma_p(G)=\gamma(G)=2$.

	\item If $\{v,w\}\cap V_1=\emptyset$, then proceed as in item (1).
	Suppose w.l.o.g. that $v\in V_1$.
	Notice that $N[v]=V(G)\setminus [(V_1-v)\cup \{w\}]$.
	If $r_1=1$, then for every $u\not\in\{v,w\}$, $u\rightarrow w$.
	Thus, $\gamma_p(G_e)=1$.
	Otherwise, suppose that $r_1=2$ and $V_1=\{v,v'\}$.
	Then, $N[v']=V(G)-v$ and for any vertex $u\not\in \{v,v',w\}$, $u \rightarrow v$.
	Hence, $\gamma_p(G_e)=1$.
	
	\item If $\{v,w\}\cap V_1=\emptyset$, then proceed as in item (2).
	Otherwise, suppose w.l.o.g. that $v\in V_1$ and $V_1=\{v,v',v''\}$.
	Notice that $N[v']=V(G)\setminus \{v,v''\}$.
	Next, observe that $w \rightarrow v''$ and for any vertex  $u\not\in \{v,v',v'',w\}$,   $u  \rightarrow v$.
	Hence, $\gamma_p(G_e)=1$.

	\item Notice that, for every $u\in V(G)$, $cl(u)=N[u]$ and $|N[u]|\le n-3$.
	Thus, $\gamma_p(G)\ge2$.
	Moreover, for every pair of vertices $\{u_1,u_2\}$ such that $ u_1 \in V_i $ and $ u_2 \in V_j $ and $\{u_1,u_2\}\cap\{v,w\}=\emptyset$, $N[\{u_1,u_2\}]=V(G)$.
	Hence, $\gamma_p(G_e)=\gamma(G_e)=2$.
	
\end{enumerate}
\vspace{-.4cm}\end{proof}
%%%%%%%%%%%%%%%%%%%%%

\vspace{.4cm}
A tree is called a \emph{spider} if it has a unique vertex of degree greater than 2.  
We define the \emph{spider number}  of a tree $T$,  denoted  by $sp(T )$, to be the minimum number of subsets into which $V (T )$ can be partitioned so that each subset induces a spider. 
%We call such a partition a \emph{spider partition} and each set of the partition a \emph{spider subset}.

\vspace{.2cm}
%%%%%%%%%%%%%%%%%%%%%%%%%%%%
\begin{theorem} [\cite{hahehehe02}]\label{sp}
	For any tree $T$,  $\gamma_p(T) = sp(T)$.
\end{theorem}
%%%%%%%%%%%%%%%%%%%%%%%%%%%%

%%%%%%%%%%%%%%%%%%%%%%%%%%%%
\begin{cor} %[\cite{hahehehe02}]\label{1gpg}
	For any tree $T$, $\gamma_p(T) = 1$ if and only if $T$ is a spider.
\end{cor}
%%%%%%%%%%%%%%%%%%%%%%%%%%%%

%%%%%%%%%%%%%%%%%%%%%%%%%%%%
\begin{theorem}[\cite{zhka07}]
If $G$ is a planar (resp. outerplanar) graph of diameter at most 2 (resp. at most 3), then $\gamma_p(G)\le 2$ (resp. $\gamma_p(G)=1$).
\end{theorem}
%%%%%%%%%%%%%%%%%%%%%%%%%%%

%%%%%%%%%%%%%%%%%%%%%%%%%%%%%%%%%%%%%%%%%%%%%%%%%%%%%%%%%%%%%%%%%%%%%%%%%%%%%%%%%%%%%%%%%%%%%%%%%%
%%%%%%%%%%%%%%%%%%%%%%%%%%%%%%%%%%%%%%%%%%%%%%%%%%%%%%%%%%%%%%%%%%%%%%%%%%%%%%%%%%%%%%%%%%%%%%%%%%
\section{Graphs with large maximum degree}\label{md}

%%%%%%%%%%%%%%%%%%%%%
\begin{prop}\label{gpn-3}
Let $G$ a graph of order $n$ and maximum degree $\Delta$.
\begin{enumerate}[{\rm (1)}]

\item If $n-2\le \Delta \le n-1$, then  $\gamma_p(G)=1$.

\item If $n-4\le \Delta \le n-3$, then  $1 \le \gamma_p(G) \le 2$.

\end{enumerate}
 
\end{prop}
\begin{proof} Let $u$ a vertex such that $deg(u)=\Delta$, that is, such that $|N[u]|=\Delta+1$.

\vspace{.2cm}{\rm (1)} If $\Delta = n-1$, then $1 \le  \gamma_p(G) \le \gamma(G)=1$, which means that $\gamma_p(G)=1$.
Let $u$ a vertex such that $deg(u)=\Delta$, that is, such that $|N[u]|=\Delta+1$.
If $\Delta = n-2$, then $|N[u]|=n-1$, i.e., there exists  a vertex $w$ such that $V(G)\setminus N[u]=\{w\}$.
Thus, for some vertex $v \in N(u)$, $v \rightarrow w$, which means that $\{u\}$ is a PD-set.

\vspace{.2cm}{\rm (2)} 
Suppose that $\Delta = n-3$.
Let $w_1,w_2\in V(G)$ such that $V(G)\setminus N[u]=\{w_1,w_2\}$.
Take the set $S=\{u,w_1\}$.
If $w_1w_2\in E(G)$, then  $S$ is a dominating set of $G$, and thus it is a power dominating set.
If $w_1w_2\not\in E(G)$, then $N[S]=V(G)\setminus \{w_2\}$.
Hence, $S$ is a power dominating set since for some vertex $v \in N(u)$, $v \rightarrow w_2$.

Finally, assume that $\Delta = n-4$.
Let $w_1,w_2,w_3\in V(G)$ such that $N=V(G)\setminus N[u]=\{w_1,w_2,w_3\}$.
We distinguish cases.

\vspace{.2cm}\noindent{\bf Case 1:} 
$G[N]$ is not the empty graph $\overline{K}_3$.
Suppose w.l.o.g. that $w_1w_2\in E(G)$.
Take the set $S=\{u,w_1\}$.
If $w_1w_3\in E(G)$, then  $S$ is a dominating set of $G$, and thus it is a power dominating set.
If $w_1w_3\not\in E(G)$, then $N[S]=V(G)\setminus \{w_3\}$.
Hence, $S$ is a power dominating set since either $w_2 \rightarrow w_3$ or, for some vertex $v \in N(u)$, $v \rightarrow w_3$.

\vspace{.2cm}\noindent{\bf Case 2:} 
$G[N]$ is  the empty graph $\overline{K}_3$.
For $i\in \{1,2,3\}$, let $v_i\in N(u)$ be such that $v_iw_i\in E(G)$.
If for every $i\in\{1,2,3\}$, $N(v_i)\cap \{w_1,w_2,w_3\}=\{w_i\}$,  then  $\{u\}$ is a dominating set of $G$, and thus it is a power dominating set.
If for some  $i\in\{1,2,3\}$, $|N(v_i)\cap \{w_1,w_2,w_3\}|\ge 2$, assume w.l.o.g. that $i=1$.
In this case, $S=\{u,v_1\}$ is a power dominating set since 
$V(G)\setminus \{w_3\} \subseteq N[S]$ and either $v_1 \rightarrow w_3$ or  $v_3 \rightarrow w_3$.
\end{proof}
%%%%%%%%%%%%%%%%%%%%%

There are graphs with maximum degree $\Delta=n-5$ such that $\gamma_p(G)\ge3$. 
The simplest example is shown in Figure \ref{f}.

\begin{figure}[!h]
	\centerline{\includegraphics[height=3cm]{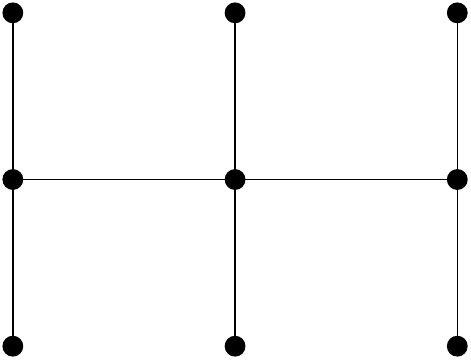}}
	\caption{A graph with maximum degree $\Delta=n-5$ and $\gamma_p(G)=3$}
	\label{f}
\end{figure}

%%%%%%%%%%%%%%%%%%%%%%%%%%%%

\begin{lemma}
	Let $G$ be a  graph of order $n\ge4$.
	Let $u,w_1,w_2 \in V(G)$ such that $deg(u)=n-3$ and $V(G)= N[u] \cup \{w_1,w_2\}$
	Then, $\{u\}$ is a PD-set if and only if  $w_1$ and $w_2$ are not twins.
\end{lemma}
\begin{proof}
Suppose first that $w_1$ and $w_2$ are  twins.
In this case, every power dominating set must contain either $w_1$ or $w_2$.
Conversely, assume that $w_1$ and $w_2$ are not twins.
If $N(w_1)=\{w_2\}$, then for some vertex $v\in N(u)$, $v \rightarrow w_1$ and $w_1 \rightarrow w_2$, which means that $\{u\}$ is a PD-set.
If $deg(w_1)\ge 2$, then take  a vertex $v_1\in N(u)$ such that $w_1\in N(v_1)$ and $w_2 \not\in N(v_1)$.
Thus, $v_1 \rightarrow w_1$ and $v_2 \rightarrow w_2$, for any vertex $v_2$ such that $w_2\in N(v_2)$.
\end{proof}
%%%%%%%%%%%%%%%%%%%%%%%%%%%

%%%%%%%%%%%%%%%%%%%%%
\begin{cor}
Let $G$ be a  graph of order $n\ge4$.
If there exists a vertex $u\in V(G)$ such that $deg(u)=n-3$ and the pair of vertices of $V(G)\setminus N[u]$ are not twins, then $\gamma_p(G)=1$.
\end{cor}
%%%%%%%%%%%%%%%%%%%%%%%%%%%

The converse of this statement is not true.
For example, if we consider the graph $G$ displayed in Figure \ref{G9}, then it is easy to check that $\{w_1\}$ is a PD-set of $G$.

%%%%%%%%%%%%%%%%%%%%%%%%%%%%%%%%%%%%
\begin{figure}[!h]
	\centerline{\includegraphics[height=5cm]{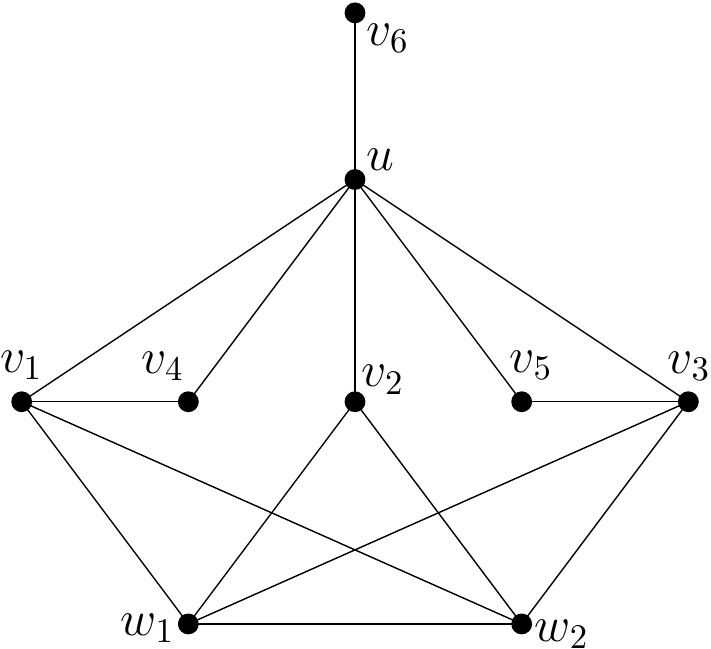}}
	\caption{The list a forcing chains for ${\cal F}_{N[w_1]}$ is: $\{(w_1),(w_2),(v_2,u,v_6),(v_1,v_4),(v_3,v_5)\}$.}
	\label{G9}
\end{figure}
%%%%%%%%%%%%%%%%%%%%%%%%%%%%%%%%%%%%

%%%%%%%%%%%%%%%%%%%%%%%%%%%%%%%%%%%%%%%
\begin{theorem}
Let $ G $ be a  $ (n-3) $-regular graph of order $n \geq 5$. 
Then, $ \gamma_{p}(G) = 1 $ if and only if there exist an edge $e=uv\in E(G) $ such that $ \mid N[v] \setminus N[u] \mid=1$. Otherwise  $ \gamma_{p}(G) = 2 $.
\end{theorem}
\begin{proof}
According to part 2 of proposition \ref{gpn-3} $ \gamma_{p}(G) \leq 2 $.
If $n=5$, then $G\cong C_5$, and the equivalence  is obvious.
Suppose thus that $n\ge6$.
\newline ($\Rightarrow$):
Let $S=\{u\}$ be a $\gamma_p$-set of $G$.
%Let $N(u)=\{v_1,\ldots,v_{n-3}\}$ and
Let $W=V(G)\setminus N[u]=\{x,y\}$.
As $S$ is a $\gamma_p$-set, there must exist a vertex $v\in N(u)$ such that $|N(v)\cap W|=1$.
Hence,   there exist a unique vertex $w\in N(u)\setminus \{v\}$ such that $w\not\in N(v)$, as $\deg(v)=n-3$.
\newline ($\Leftarrow$):
Take the sets $S=\{u\}$ and $W=V(G)\setminus N[u]=\{x,y\}$.
As $ \mid N[v] \setminus N[u] \mid=1$ and $deg(u)=deg(v)=n-3$, $|N(v)\cap W|=1$.
Hence, if for example $N(v)\cap W=\{x\}$, then $v \rightarrow x$, which means that $S$ is a $\gamma_p$-set of $G$.
\end{proof}
%%%%%%%%%%%%%%%%%%%%%%%%%%%%%%%%%%%%%%%%%%%

Here we define two families of graphs;

$\bullet  $ Let three vertices $ u $ and $ w $ and $ z $ induce a path in $ G $(w is in the middle vertex) such that $ u $ or $ z $ is the difference with $ w $ in exactly one neighbor, i.e.,
\[ \exists x \in \{u,z\} ~~ s.t ~~ \mid N[w] \setminus N[x] \mid =1 \]
This induce subgraph is called $ \mathbb{P}_{3} $.

 $\bullet  $ For another family, suppose three vertices $ u $ and $ w $ and $ z $ induce a cycle in $ G $ such that at least two of them are the difference with $ w $ in exactly one neighbor, i.e.,
\[ \exists x,y \in \{u,w,z\} ~~ s.t ~~ \mid N[x] \setminus N[y] \mid =1 \]
This induce subgraph is called $ \mathbb{C}_{3} $.

\begin{theorem}
Let $ G $ be a  $ (n-4) $-regular graph of order $n \geq 6 $. 
$ \gamma_{p}(G) = 1 $ if and only if $ G $ contains $ \mathbb{P}_{3} $ or $ \mathbb{C}_{3} $ as a induce subgraph. 
\end{theorem}

\begin{proof}
($\Rightarrow$):Let $S=\{v\}$ be a $\gamma_p$-set of $G$ and $V(G)\setminus N[v]=\{u,w,z\}$. Now we study four states;

\begin{itemize}
\item There is a path between $ u, w $ and $ z $.
\item only two of $\lbrace u, w, z \rbrace$ are adjacent.
\item $ (u, w) $, $ (w, z) $ and $ (z, u) \in E(G) $. 
\item None of them is adjacent.
\end{itemize}

%(a) There is a path between $ u, w $ and $ z $.\\
In the First state, Without losing generality, let $ (u, w)$ and $(w, z) \in E(G) $. $ w $ is not adjacent to two neighbors of $ v $, which we call $ w_{1} $ and $ w_{2} $(because $ G $ is a $ (n-4)- $regular graph and every vertex is not adjacent to 3 vertices). $ u $ and $ z $ also is not adjacent to one neighbor of $ v $ called $ u_{1} $ and $ z_{1} $. Now we have five cases;\\

\textbf{Case 1:} $ u_{1}=z_{1} $.
 So $ u $ and $ z $ are twin, then should $ u $ or $ z \in S $. therefore  $ \gamma_{p}(G) > 1 $ and this is a contradiction.

\textbf{Case 2:} $ z_{1}=w_{1} $ (or $ u_{1}=w_{1} $ ).
So $ z_1 \rightarrow u $ and next $ u \rightarrow w $. Then $ \gamma_p(G)=1 $ in this case.

\textbf{Case 3:} $ z_{1}=w_{1} $ and $ u_{1}=w_{2} $.
This case is like the previous case.

\textbf{Case 4:} None of them is not equal.
 In this case, each of neighbors of $ v $ is adjacent to at least two of $ \lbrace u, w, z \rbrace $. So we can not continue. Then $ \gamma_{p}(G) > 1 $ and so this case is not accepted.

\textbf{Case 5:} $ z_{1}=u_{1}=w_{1} $.
 $ z_{1}=u_{1}=w_{1} $ is not adjacent to any one of $ \lbrace u, w, z \rbrace $ and other neighbors of $ v $ are  adjacent to at least two of $ \lbrace u, w, z \rbrace $. So we can not continue and then $ \gamma_{p}(G) > 1 $.

Thus only the second and third cases in the first state are acceptable that in both of them, $ \{ u, w, z \} $ induces $ \mathbb{P}_{3} $ in $ G $.

%(b) only two of $\lbrace u, w, z \rbrace$ are adjacent.\\
In the second state, Without losing generality, let $ (u, w) \in E(G) $. So $ N(z)=N(v) $ but $ u $ and $ w $ are not adjacent to one neighbor of $ v $ called $ u_{1} $ and $ w_{1} $. If $ u_{1}=w_{1} $, then $ u $ and $ w $ are twin and we can not continue. Now if $ u_{1} \neq w_{1} $, then In this case, each of neighbors of $ v $ are  adjacent to at least two of $ \lbrace u, w, z \rbrace $. So we can not continue propagation. So in both cases, $ \gamma_{p}(G) > 1 $.

%(c) $ (u, w) $, $ (w, z) $ and $ (z, u) \in E(G) $. \\
In the third state, $ u $, $ w $ and $ z $ are not adjacent to two neighbors of $ v $, which we call $ u_{1} $, $ u_{2} $, $ w_{1} $, $ w_{2} $, $ z_{1} $ and $ z_{2} $ respectively. Now we have eight cases;\\

\textbf{Case 1:} None of them is not equal.

\textbf{Case 2:} $ w_{1}=u_{1}=z_{1} $ and $ w_{2}=u_{2}=z_{2} $.

\textbf{Case 3:} $ w_{1}=u_{1}=z_{1} $ and $ w_{2}=u_{2} $.

\textbf{Case 4:} $ w_{1}=u_{1}=z_{1} $.

\textbf{Case 5:} $ w_{1}=u_{1} $.

\textbf{Case 6:} $ w_{1}=u_{1} $ and $ u_{2}=z_{2} $.

\textbf{Case 7:} $ w_{1}=u_{1} $ and $ w_{2}=u_{2} $.

\textbf{Case 8:} $ w_{1}=u_{1} $ and $ w_{2}=z_{2} $ and $ z_{1}=u_{2} $.

Like argument of the second state, it can easily be seen that there are contradiction in the second, third, fourth and seventh cases. Also $ \{u, w, z\} $ induses $ \mathbb{C}_{3} $ in the fifth and sixth and eighth cases.

In the fourth state, In this case,  $ u, w, z $ are twin pairwise. Then $ \gamma_{p}(G) > 1 $.

 ($\Leftarrow$):It is clear.
\end{proof}

An example for the fourth state is $ K_{\underbrace{4, 4, \cdots , 4}_{\frac{n}{4}}} $ that  $ \gamma_{p}(K_{\underbrace{4, 4, \cdots , 4}_{\frac{n}{4}}})=2 $.

The following corollary is obtained directly from the above theorem.

\begin{cor}
Let $ G $ be a  $ (n-4) $-regular graph of order $n \geq 6 $. 
$ \gamma_{p}(G) = 2 $ if and only if $ G $ contains neither $ \mathbb{P}_{3} $ nor $ \mathbb{C}_{3} $ as a induce subgraph. 
\end{cor}

We know that '' in what condition is $ \gamma_p $ of a graph equal to 1 '' is a open problem. In the next sections, we obtain it for some families of graphs.

%%%%%%%%%%%%%%%%%%%%%%%%%%%%%%%%%%%%%%%%%%%%%%%%%%%%%%%%%%%%%%%%%%%%%%%%%%%%%%%%%%%%%%%%%%%%%%%%%%%%%%%%%%%%%%%%%%%%%%%%%%%%%%%%%%%%%%%%%%%%%%%%%%%%%%%%%%%%%%%%%%%%%%%%%%%%%%%%%%%%%%%%%%%%%%%%%%%%%%%%%%%%%%%%%%%%%%%%%%%%%%%%%%%%%%%%%%%%

\section{Mycieleskian graphs}\label{lp}

\begin{defi}[\cite{va11}]\label{MG}
Let $G = (V , E)$ be a graph with the vertex set $V = \{ v_i \vert 1 \leq i \leq n \}$. The \emph{Mycieleskian graph} $\mu(G)$ of a graph $G$ is a graph with the vertex set $V \cup U \cup \{ w \}$ such that $U = \{ u_i \vert 1 \leq i \leq n \}$, and the edge set $E \cup \{u_i v_j \vert v_i v_j \in E(G) \} \cup \{u_i w | u_i \in U \}$.
\end{defi}

\begin{theorem}\cite{va11}
If $G$ be a connected graph, then $\gamma_p ( \mu(G)) \in \{ 1, \gamma_G, \gamma_p(G) + 1 \} $.
\end{theorem}

%%%%%%%%%%%%%%%%%%%%%%%%%%%%%%%%%%%%%%%%%%%%%%%%%%%%

\begin{theorem}\cite{va11}\label{MG=1}
Let $ G $ be a connected graph. If $ G $ has one universal vertex, then $ \gamma_p(\mu(G)) =1 $.
\end{theorem}

Notice that according to theorem \ref{PP=1}, inverse of theorem \ref{MG=1} is not true.
%%%%%%%%%%%%%%%%%%%%%%%%%%%%%%%%%%%%%%%%%%%%%
Next, we calculate power domination number of mycieleskian of the graphs in the table 1.
In the first, you can see the direct result of the above theorem.

\begin{cor}
$ \gamma_p(\mu(K_n))=\gamma_p(\mu(K_{1,n}))=\gamma_p(\mu(W_n))=1 $
\end{cor}

%%%%%%%%%%%%%%%%%%%%%%%%%%%%%%%%%%%%%%%%%%%%%%%%%%%%

\begin{theorem}\label{PP=1}
$\gamma_p(\mu(P_n))=1 $.
\end{theorem}
\begin{proof}
Let $ G=P_n $ and $ V(G)= \{ v_1, \cdots , v_n \} $. We claim that $ S=\{v_2\} $ is a PD-set for $ M(G) $ because $ \{ v_1, v_3, u_1, u_3 \} $ are monitored in d.s(domination step) and $ \{ w, u_2 \} $ are monitored in p.s(propagation steps). Also $ \{ v_4 \}, \{ u_4 \}, \{ v_5 \}, \{ u_5 \}, \cdots, \{ v_n \}, \{ u_n \} $ are monitored in next propagation steps.
\end{proof}

%%%%%%%%%%%%%%%%%%%%%%%%%%%%%%%%%%%%%%%%%%%%%%%%%%%%

\begin{theorem}
If $ G=C_n $, and
\begin{equation*}
\gamma_p(\mu(G))=
\begin{cases}
1 & n=3 \\
2  & n \geq 4.
\end{cases}
\end{equation*}
\end{theorem}
\begin{proof}
Let $ G=C_n $ and $ V(G)=\{ v_1, \cdots, v_n \} $. For $ n=3 $, according to theorem \ref{MG=1} $ \gamma_p(\mu(G))=1 $. For $ n \geq 4 $, $ S= \{ v_1, w \} $ is a PD-set for $ \mu(G) $ so $ \gamma_p(\mu(G)) \leq 2 $. Now we prove $ \gamma_p(\mu(G)) \neq 1 $. If $ \gamma_p(\mu(G))= 1 $, then we have three state;\\
\textbf{State 1:} $ S= \{ v_i \} $.\\
\textbf{State 2:} $ S= \{ u_i \} $.\\
\textbf{State 3:} $ S= \{ w \} $.\\
In state 1, $ \{ v_{i-1}, v_{i+1}, u_{i-1}, u_{i+1} \} $ are monitored in d.s but p.s can not occur because all of $ \{ v_{i-1}, v_{i+1}, u_{i-1}, u_{i+1} \} $ are adjacent to at least two vertices of $ V(M(G)) \setminus \{ v_{i-1}, v_{i+1}, u_{i-1}, u_{i+1} \} $.\\
In state 2, $ \{ w, v_{i-1}, v_{i+1} \} $ are monitored in d.s and again like state 1, p.s can not occur.\\
In state 3, $ \{ u_1, \cdots, u_n \} $ are monitored in d.s but p.s can not occur.\\
 So $ \gamma_p(\mu(G)) \neq 1 $ and then $ \gamma_p(\mu(G))= 2 $.
\end{proof}

%%%%%%%%%%%%%%%%%%%%%%%%%%%%%%%%%%%%%%%%%%%%%%%%%%%%

\begin{theorem}
$ \gamma_p(\mu(K_{h,n-h}))=2 $, for $ h \geq 2, n \geq 4 $.
\end{theorem}
\begin{proof}
Let $ G=K_{h,n-h} $ and $ V(G)=\{ v_1, \cdots , v_h \} \cup \{ v_{h+1}, \cdots , v_n \} $. If $ v_i \in \{ v_1, \cdots , v_h \} $ and $ v_j \in  \{ v_{h+1}, \cdots , v_n \} $, then $ S=\{ v_i, v_j \} $ is a PD-set for $ M(G) $ because $ V(\mu(G)) \setminus \{ w \} $ are monitored in d.s and $ w $ is monitored in p.s. So $ \gamma_p(\mu(G)) \leq 2 $. Now we claim $ \gamma_p(\mu(G)) \neq 1 $. If $ \gamma_p(\mu(G))= 1 $, then we have five state;\\
\textbf{State 1:} $ S= \{ v_i \} $, for $ 1 \leq i \leq h $.\\
\textbf{State 2:} $ S= \{ v_j \} $, for $ h+1 \leq i \leq n $.\\
\textbf{State 3:} $ S= \{ u_i \} $, for $ 1 \leq i \leq h $.\\
\textbf{State 4:} $ S= \{ u_j \} $, for $ h+1 \leq i \leq n $.\\
\textbf{State 5:} $ S= \{ w \} $.\\
In all state, p.s can not occur. Therefore $ \gamma_p(\mu(G))= 2 $.
\end{proof}

%%%%%%%%%%%%%%%%%%%%%%%%%%%%%%%%%%%%%%%%%%%%%%%%%%%%
A family graph closely related to Mycieleskian graph is called the shadow graph. The shadow graph $S(G)$ of a graph $G$ is the graph obtained from $G$ by adding a new vertex $ u_i $ for each vertex $ v_i $ of G and joining $ u_i $ to the neighbors of $ v_i $ in G \cite{gash99}. The vertex $ u_i $ is called the shadow vertex of $ v_i $ and we have called the vertex $ v_i $ the orginal vertex of $ u_i $.

\begin{theorem}\label{aaaaa}
If $G$ be a connected graph, then $ \gamma_p(G) \leq \gamma_p(S(G)) \leq 2 \gamma_p(G) $.
\end{theorem}
\begin{proof}
Let $ G $ is a graph with $ V(G) =\{ v_1,\cdots, v_n \} $ and $ V(S(G))= V(G) \cup \{ u_1, \cdots, u_n \} $. If $ S_{S(G)}= Y \cup W $ is a $ \gamma_p $-set for $ S(G) $ such that $ Y \subseteq V(G) $ and $ W \subseteq  \{ u_1, \cdots, u_n \} $, then $ S_G = Y \cup Y' $ is a PD-set for $ G $ that $ Y' \subseteq V(G) $ and $ Y' $ contains the orginal vertices of $ W $. So $ \gamma_p(G) \leq \gamma_p(S(G)) $.\\
In other hand, if $ S_G $ be a $ \gamma_p $-set for $ G $, then $ S_{S(G)}=S_{G} \cup S^{'}_{G} $ is a PD-set for $ S(G) $ that $ S^{'}_{G} $ contains the shadow vertices of $ S_G $.

\begin{figure}[!h]
	\centerline{\includegraphics[height=3.5cm]{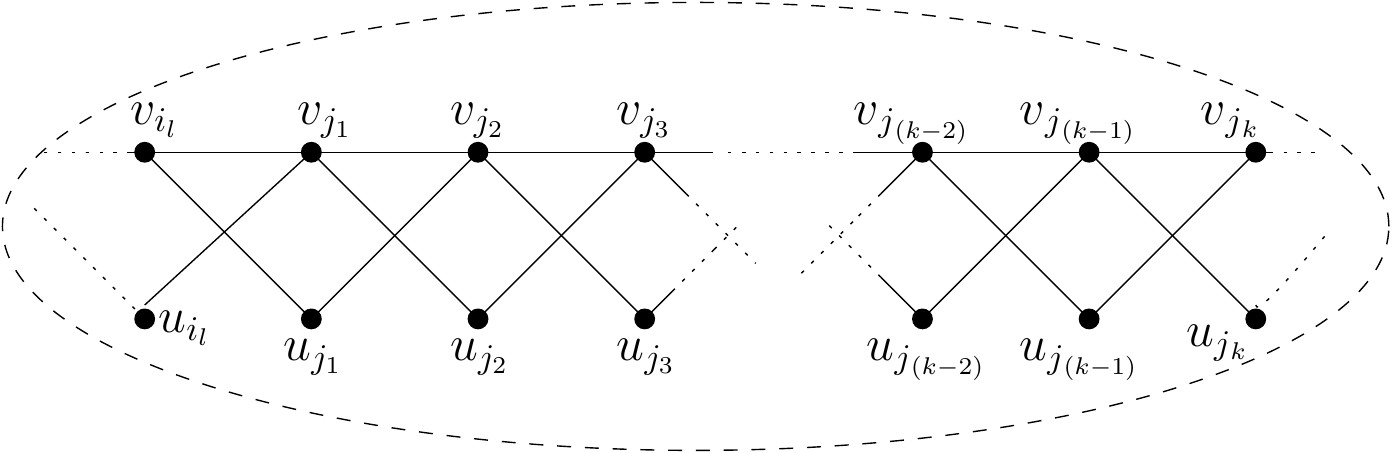}}
	\caption{Part of the graph $ G $}
	\label{partG}
\end{figure}

because suppose that $ v_i \in S_G $ and $ N(v_i)=\{ v_{i_1},\cdots, v_{i_k} \} $. in d.s, $ N(v_i) \cup \{ u_{i_1}, \cdots, u_{i_k} \} $ are monitored in $ S(G) $. Now according to figure \ref{partG}, if propagation accures from $ v_{i_l} $ and forcing chain be as $ ( v_{i_l} ,v_{j_1},\cdots, v_{j_k}) $ in $ G $, then propagation accures from $ u_{i_l} $ and forcing chains are as $ (u_{i_l}, v_{j_1}) $, $ (v_{i_l}, u_{j_1}, v_{j_2}) $, $ (v_{j_1}, u_{j_2},v_{j_3}) $, $ \cdots $, $ (v_{j_{(k-2)}}, u_{j_{(k-1)}}, v_{j_k}) $. Thus $ \gamma_p(S(G)) \leq 2 \gamma_p(G) $.
\end{proof}

\begin{lemma}
If $ S= \{ v_{i_1},\cdots, v_{i_k} \} $ be a $ \gamma_p $-set for a graph $ G $ in which every vertex $ v_i $;
\item[(1)] has been at least one neighbor $ v_j \in \{ V(G) \setminus S \} $ such that $ N[v_j] \subseteq N[v_i] $ or,
\item[(2)]  has been at least one neighbor $ v_j \in S $,\\
 then $ \gamma_p(G) = \gamma_p(S(G)) $.
\end{lemma}
\begin{proof}
If $ S= \{ v_{i_1},\cdots, v_{i_k} \} $ is a $ \gamma_p $-set for $ G $ and $ v_{i_l} \in S $ such that it has been one neighbor $ v_j \in \{ V(G) \setminus S \} $, then in p.s, $ u_{i_l} $ is monitored by $ v_j $($ u_{i_l} $ is the shadow vertex of $ v_{i_l} $) in S(G).\\
In other hand, if $ v_{i_l} $ has been on neighbor $ v_j \in S $, then in d.s $ u_{i_l} $ is monitored by $ v_j \in S(G) $. So like the proof of theorem \ref{aaaaa}, we have $ \gamma_p(S(G)) \leq  \gamma_p(G) $ and then $ \gamma_p(S(G)) = \gamma_p(G) $.
\end{proof}

\begin{cor}
If $ S = \{ v_1, \cdots, v_k \} $ be a $ \gamma_p $-set for a graph $ G $ and $ A \subseteq S $ that every vertex $ v_i $ in $ A $;
\item[(1)] has been any neighbor $ v_j \in \{ V(G) \setminus S \} $ such that $ N[v_j] \subseteq N[v_i] $ or,
\item[(2)] has been any neighbor $ v_j \in S $,\\
 then $ \gamma_p(S(G)) \leq 2 \gamma_p(G) - \mid A \mid $.
\end{cor}

\begin{theorem}\label{SG=1}
Let $ G $ be a connected graph. If $ G $ has one universal vertex, then $ \gamma_p(S(G)) =1 $.
\end{theorem}
\begin{proof}
Let $ G $ is a graph that $ v_1 $ is an universal vertex of $ G $ and also $ V(G) = \{ v_1, \cdots, v_n \} $. We claim $ S=\{ v_1 \} $ is a $ \gamma_p $-set for $ S(G) $, because $ V(S(G)) \setminus \{ u_1 \} $ are monitored in d.s and $ \{ u_1 \} $ is monitored in p.s.
\end{proof}

In the following, you can see the direct result of the above theorem.

\begin{cor}
$ \gamma_p(S(K_n))=\gamma_p(S(K_{1,n}))=\gamma_p(S(W_n))=1 $
\end{cor}

\begin{theorem}
$ \gamma_p(S(P_n))=1 $
\end{theorem}
\begin{proof}
Let $ G=P_n $ that $ V(G)= \{ v_1, \cdots , v_n \} $. We claim that $ S=\{v_2\} $ is a $ \gamma_p $-set for $ S(G) $ because $ \{ v_1, v_3, u_1, u_3 \} $ are monitored in d.s and other vertices are monitored in several propagation steps.
\end{proof}

\begin{theorem}
If $ G=C_n $, then
\begin{equation*}
\gamma_p(S(G))=
\begin{cases}
1 & n = 3 \\
2  & n \geq 4.
\end{cases}
\end{equation*}
\end{theorem}
\begin{proof}
Let $ G=C_n $ and $ V(G)=\{ v_1, \cdots, v_n \} $. For $ n=3 $, according to theorem \ref{SG=1} $ \gamma_p(S(C_3))=1 $. For $ n \geq 4 $, $ S= \{ v_1, u_1 \} $ is a PD-set for $ S(G) $ so $ \gamma_p(S(G)) \leq 2 $. Now we prove $ \gamma_p(S(G)) \neq 1 $. If $ \gamma_p(S(G))= 1 $, then we have two state;\\
\textbf{State 1:} $ S= \{ v_i \} $, for $ 1 \leq i \leq n $.\\
\textbf{State 2:} $ S= \{ u_i \} $, for $ 1 \leq i \leq n $.\\
In state 1, $ \{ v_{i-1}, v_{i+1}, u_{i-1}, u_{i+1} \} $ are monitored in d.s and $ \{ v_{i-2}, v_{i+2} \} $ are monitored  in p.s and p.s can not occur because all of $ \{ v_{i-2}, v_{i+2},  v_{i-1}, v_{i+1}, u_{i-1}, u_{i+1} \} $ are adjacent to at least two vertices of $ V(S(G)) \setminus \{  v_{i-2}, v_{i+2}, v_{i-1}, v_{i+1}, u_{i-1}, u_{i+1} \} $.\\
In state 2, $ \{v_{i-1}, v_{i+1} \} $ are monitored in d.s but p.s can not occur.
\end{proof}

\begin{theorem}
$ \gamma_p(S(K_{h,n-h}))= 2 $ , for $ h \geq 2, n \geq 4 $.
\end{theorem}
\begin{proof}
Let $ G=K_{h,n-h} $ and $ V(G)=\{ v_1, \cdots , v_h \} \cup \{ v_{h+1}, \cdots , v_n \} $. If $ v_i \in \{ v_1, \cdots , v_h \} $ and $ v_j \in  \{ v_{h+1}, \cdots , v_n \} $, then $ S=\{ v_i, v_j \} $ is a PD-set for $ S(G) $ because $ V(S(G)) $ are monitored in d.s. So $ \gamma_p(S(G)) \leq 2 $. Now we claim $ \gamma_p(S(G)) \neq 1 $. If $ \gamma_p(S(G))= 1 $, then we have four state;\\
\textbf{State 1:} $ S= \{ v_i \} $, for $ 1 \leq i \leq h $.\\
\textbf{State 2:} $ S= \{ v_j \} $, for $ h+1 \leq i \leq n $.\\
\textbf{State 3:} $ S= \{ u_i \} $, for $ 1 \leq i \leq h $.\\
\textbf{State 4:} $ S= \{ u_j \} $, for $ h+1 \leq i \leq n $.\\
In all state, it is clear that all vertices are not monitored. Therefore $ \gamma_p(S(G))= 2 $.
\end{proof}
%%%%%%%%%%%%%%%%%%%%%%%%%%%%%%%%%%%%%%%%%%%%%%%%%%%%%%%%%%%%%%%%%%%%%%%%%%%%%%%%%%%%%%%%%%%%%%%%%%%%%%%%%%%%%%%%%%%%%%%%%%%%%%%%%%%%%%%%%%%%%%%%%%%%%%%%%%%%%%%%%%%%%%%%%%%%%%%%%
%%%%%%%%%%%%%%%%%%%%%%%%%%%%%%%%%%%%%%%%%%%%%%%%%%%%%%%%%%%%%%%%%%%%%%%%%%%%%%%%%%%%%%%%%%%%%%%%%%

\section{Central graphs}\label{dp}

\begin{defi}[\cite{ve07}]\label{c=1}
The central graph $C(G)$ of a graph $G$ of order $n$ and size $m$ is a graph of order $n + m$ and size $(^{n}_{2}) + m$ which is obtained by subdividing each edge of $G$ exactly once and joining all the non-adjacent vertices of $G$ in $C(G)$.
\end{defi}

\emph{The anti-cycle vertex} $ v $ of a graph $ G $ is a vertex such that $ G[V(G) \setminus \{ v \}] $ is an acyclic (a graph having no cycle). Note that $ G[V(G) \setminus \{ v \}] $ can be connected or disconnected.

\begin{theorem}\label{cc=1}
Let $ G $ is connected graph. If $ G $ has at least one anti-cycle vertex, then $ \gamma_p(C(G)) =1 $.
\end{theorem}
\begin{proof}
Let $ G $ be a graph that $ V(G) = \{ v_1, \cdots , v_n \} $ and
 \[ V(C(G))=\{ v_1, \cdots , v_n \} \cup \{u_{i,j} \mid v_i v_j \in E(G) ~~for~ some~ 1 \leq i , j \leq n \} \] 
and also $ v_k $ be an anti-cycle vertex of $ G $. We prove that $ S=\{ v_k \} $ is a PD-set for $ C(G) $. At the first, we consider the state that $ V(G) \setminus \{ v_k \} $ is connected because when it isn't connected will also be also proven like this state. Now let $ deg_{G}(v_k)=s $ and without losing generality suppose that $ N_G (v_k)=\{ v_1, v_2, \cdots , v_s \} $, so $ v_k $ is not adjacent to $ \{ v_{s+1}, \cdots , v_n \} $. In the other hand 
\[ N_{C(G)} [v_k]= \{ v_1, v_2, \cdots , v_s, v_k, u_{k(s+1)}, \cdots, u_kn \}, \]
then $ \{ v_{s+1}, \cdots , v_n \} $ are monitored in p.s. Notice that $ V(G) \setminus \{ v_k \} $ induce a tree and all of $ V(G) \setminus \{ v_k \}  $ are monitored. So the propagation can easily be continued from the leaves. Thus in the next p.s, $ \{u_{ij}\}_{i,j} $ (for some $ 1 \leq i,j \leq n $ and $ i,j \neq k $) are monitored.
\end{proof}

%%%%%%%%%%%%%%%%%%%%%%%%%%%%%%%%%%%%%%%%%%%%%%%%%%

Inverse of theorem \ref{c=1} is not true. For example $ K_{3,3} $ has no anti-cycle vertex but $ \gamma_p(C(K_{3,3}))=1 $(proved in theorem \ref{CKmn}).\\
Next, we calculate power domination number of central of the graphs in the table 1.
In the following, you can see the direct result of theorem \ref{cc=1}.

%%%%%%%%%%%%%%%%%%%%%%%%%%%%%%%%%%%%%%%%%%%%%%%%%%

\begin{cor}
$ \gamma_p(C(P_n))= \gamma_p(C(C_n))=\gamma_p(C(K_{1,n}))=\gamma_p(C(K_{2,n}))= 1 $
\end{cor}

%%%%%%%%%%%%%%%%%%%%%%%%%%%%%%%%%%%%%%%%%%%%%%%%%%

\begin{theorem}\label{CKmn}
$ \gamma_p(C(K_{h,n-h}))= 1 $ , for $ h \geq 3, n \geq 6 $.
\end{theorem}
\begin{proof}
Let $ G=K_{h,n-h} $ and $ V(G)=\{ v_1, \cdots , v_h \} \cup \{ v_{h+1}, \cdots , v_n \} $.  $ S=\{ v_k \} $ is a $ \gamma_p $-set for $ C(G) $, such that $ v_i $ is an arbitrary vertex of $ G $, bacause $ \{ v_1, \cdots , v_h \} $ and $  \{ v_{h+1}, \cdots , v_n \} $ induce $ K_h $ and $ K_{n-h} $ , respectively, in $ C(G) $. If $ v_k \in \{ v_1, \cdots , v_k \} $, then
\[ \{ v_1, \cdots , v_k \} \cup \{u_{i,j} \mid v_i v_j \in E(G) ~~for~  1 \leq i \leq k ~ and ~ h+1 \leq j \leq n \} \]
are monitored in d.s. Finally $ \{ v_{h+1}, \cdots , v_n \} $ are monitored in p.s.
\end{proof}
%%%%%%%%%%%%%%%%%%%%%%%%%%%%%%%%%%%%%%%%%%%%%%%%%%

\begin{theorem}
$ \gamma_p(C(W_n))=2. $
\end{theorem}
\begin{proof}
Let $ G=W_n $ and $ V(G) = \{ v_1, \cdots , v_n \} $ such that $ v_1 $ be the universal vertex in $ G $ and also $ V(C(G)) = V(G) \cup \{u_{i,j} \mid v_i v_j \in E(G) ~~for~some~  1 \leq i,j \leq n \} $. In the first, we prove that $ \gamma_p(C(G)) \neq 1 $. If $ \gamma_p(C(G)) = 1 $, then we have three states for $ \gamma_p $-set of $ C(G) $;\\
\textbf{State 1:} $ \gamma_p(C(G))=\{ u_{ij} \} $ for some $ 1 \leq i,j \leq n $.\\
\textbf{State 2:} $ \gamma_p(C(G))=\{ v_1 \} $.\\
\textbf{State 3:} $ \gamma_p(C(G))=\{ v_i \} $ such that $ i \neq 1 $.\\
In state 1, $ \{ v_i, v_j \} $ are monitored in d.s but  p.s can not occur because $ deg_{C(G)}(v_i)=deg_{C(G)}(v_j)=n \geq 4 $.\\
In state 2, $ \{ u_{12}, u_{13}, \cdots, u_{1n} \} $ are monitored in d.s and $ \{ v_2, \cdots , v_n \} $ are also monitored in p.s but p.s can not continue, because every vertex of $ \{ v_2,\cdots, v_n \} $ adjacent to two vertices of $ \{ u_{ij} \}_{i,j \neq 1} $.\\
In state 3, without losing generality let $ S=\{ v_2\} $. So $ \{u_{21}, u_{23}, u_{2n}, v_4, v_5, \cdots , v_{n-1} \} $ are monitored in d.s and $ \{ v_1, v_3 , v_n \} $ are monitored in p.s. The propagation can not continue because every vertex of $ \{ v_3, \cdots , v_n \} $ adjacent to at least two vertices of $ \{ u_{ij} \}_{i,j \neq 2} $.\\
It can easily be seen that $ \{ v_1 , v_i \} $($ 2 \leq i \leq n $) is a $ \gamma_p $-set for $ C(G) $. So $ \gamma_p(C(G))=2 $.
\end{proof}
%%%%%%%%%%%%%%%%%%%%%%%%%%%%%%%%%%%%%%%%%%%%%%%%%%

\begin{theorem}
$ \gamma_p(C(K_n))=n-2 $
\end{theorem}
\begin{proof}
Let $ G=K_n $ and $ V(G) = \{ v_1, \cdots , v_n \} $ and also $ V(C(G)) = V(G) \cup \{u_{i,j} \mid v_i v_j \in E(G) , 1 \leq i,j \leq n \} $. Suppose that $ S $ is a $ \gamma_p $-set of $ C(G)) $, we prove that if $ \mid S \mid=n-3 $, then all vertices of $ C(G)) $ are not monitored. In the best state, $ S $ can be considered as $ \{ v_1, \cdots , v_{n-3} \} $. In d.s, $ \{ u_{ij} \}_{i,j} $(such that $ i,j \notin \{ n, n-1, n-2 \} $) are monitored and in p.d, all vertices of $ C(G) $ are monitored except $ \{ u_{(n-2)n}, u_{(n-2)(n-1)}, u_{(n-1)n} \} $. It is clear that the propagation can not continue, so $ \mid S \mid > n-3 $.
In the other hand, If $ S=\{ v_1, \cdots, v_{n-2} \} $, then all vertices of $ C(G) $ are monitored in d.s and p.s.
\end{proof}
%%%%%%%%%%%%%%%%%%%%%%%%%%%%%%%%%%%%%%%%%%%%%%%%%%%%%%%%%%%%%%%%%%%%%%%%%%%%%%%%%%%%%%%%%%%%%%%%%%
%%%%%%%%%%%%%%%%%%%%%%%%%%%%%%%%%%%%%%%%%%%%%%%%%%%%%%%%%%%%%%%%%%%%%%%%%%%%%%%%%%%%%%%%%%%%%%%%%%
\section{Middle graphs}

\begin{defi}[\cite{thpr17}]
The middle graph $M(G)$ of a graph $G$ whose vertex set is $ V(G) \cup E(G) $ where two vertices are adjacent if and only if they are either adjacent edges of $G$ or one is a vertex and the other is an edge incident with it. 
\end{defi}

\begin{theorem}[\cite{thpr17}]
$ \gamma_p(M(P_n))=\lceil \frac{n-1}{3} \rceil $.
\end{theorem}

\begin{theorem}[\cite{thpr17}]
$ \gamma_p(M(C_n))=\lceil \frac{n}{3} \rceil  $.
\end{theorem}

\begin{theorem}[\cite{thpr17}]
$ \gamma_p(M(K_{1,n}))=1 $
\end{theorem}

\begin{defi}[\cite{vapa14}]
Let $ G= (V, E) $ be a graph. A subset $ F \subseteq E $ is an edge dominating set if each edge in $ E $ is either in $ F $ or is adjacent to an edge in $ F $. An edge dominating set $ F $ is called a minimal edge dominating set (or MEDS) if no proper subset $ F' $ of $ F $ is an edge dominating set. The edge domination number $ \gamma^{'}(G) $ is the minimum cardinality among all minimal edge dominating sets.
\end{defi}

\begin{theorem}\label{bbbb}
For any graph $ G $, $ \gamma_p(M(G)) \leq \gamma^{'}(G). $ 
\end{theorem}
\begin{proof}
Let $ G $ be a graph and $ V(G)=\{ v_1, \cdots, v_n \} $ and $ E(G)=\{ u_{ij} \mid v_i ~ and~ v_j~ are~ adjacent \} $. Without losing generality, suppose that $ u_{12} $ is in $ \gamma^{'} $-set of $ G $ and it dominate $ \{ u_{1i_{1}}, \cdots, u_{1i_{s}}, u_{2j_{1}}, \cdots, u_{2j_{k}} \} $ in $ G $. So if $ u_{12} $ be in $ \gamma_p $-set of $ M(G) $, then it monitore $ \{ u_{1i_{1}}, \cdots, u_{1i_{s}}, u_{2j_{1}}, \cdots, u_{2j_{k}} \} \cup \{ v_1,v_2 \} $ in d.s and $ \{ v_{i_{1}}, \cdots, v_{i_{s}}, v_{j_{1}}, \cdots, v_{j_{k}} \} $ are monitored in p.s. Therefore if all vertices of $ \gamma^{'} $-set be as a $ \gamma_p $-set of $ M(G) $, then all vertices of $ M(G) $ are monitored. So $ \gamma_p(M(G)) \leq \gamma^{'}(G) $.
\end{proof}

\begin{theorem}\label{m=1}
For any graph $ G $, $ \gamma_p(M(G)) =1 $ if and only if $ G $ has one universal edge(the edge that is adjacent to all other edges).
\end{theorem}
\begin{proof}
Let $ G $ is a graph and $ V(G)=\{ v_1, \cdots, v_n \} $ and $ E(G)=\{ u_{ij} \mid v_i ~ and~ v_j~ are~ adjacent \} $.\\
$ (\Leftarrow) $ If $ G $ has one universal edge i.e. $ \gamma^{'}=1 $, then according to theorem \ref{bbbb}, $ \gamma_p(M(G)) =1 $.\\
$ (\Rightarrow) $ Let $ \gamma_p(M(G)) =\{ v \} $. We have two state;\\
\textbf{State 1:} $ v \in V(G) $.\\
\textbf{State 1:} $ v \in E(G) $.\\
In state 1, without losing generality, let $ v=v_1 $ and $ N_G(v_1)=\{ v_{j_1}, \cdots, v_{j_l} \} $. So $ \{ u_{1j_1}, \cdots, u_{1j_l} \} $ are monitored in d.s but propagation can occur unless all vertices of $ N_G(v_1) $ have order 1. So in this state, $ G $ can only be a $ K_{1,n} $ and we know $  K_{1,n} $ has the universal edge.\\
In state 2, without losing generality, let $ v=u_{12} $. (proof by contradiction) Suppose that $ v_1 v_2 $ is not an universal edge in $ G $. So there is $ v_i v_j \in E(G) $ such that $ i , j \notin \{ 1 , 2 \} $. If $  v_i v_j $ be a leaf and $ v_j $ adjacent to $ v_1(or~ v_2) $ in $ G $, then $ u_{1j} $ adjacent to $ \{ u_{12}, u_{ij}, v_1, v_j \} $ in $ M(G) $. So propagation can not occur from $ u_{1j} $. Then $ v_j $ and $ u_{i,j} $ and $ v_i $ are not monitored. Thus in this case, $ \gamma_p(M(G)) \neq 1 $ and it is a contradiction.\\
If $v_i$ and $v_j$ are adjacen to $ v_1(or v_2) $ in $ G $, then  $v_i, v_j$ and $ u_{ij} $ cannot monitored in $ M(G) $. So  $ \gamma_p(M(G)) \neq 1 $ and it is a contradiction. Therefore $ v_1 v_2 $ is a universal edge and $ \gamma^{'}(G)=1 $.
\end{proof}

\begin{theorem}[\cite{AS18}]\label{W=12}
$ \gamma^{'}(M(W_n)) = 1 + \lceil \frac{n-3}{3} \rceil $.
\end{theorem}

\begin{theorem}
$ \gamma_p(M(W_n)) = 1 + \lceil \frac{n-3}{3} \rceil $.
\end{theorem}
\begin{proof}
According to theorem \ref{bbbb} and theorem \ref{W=12}, $ \gamma_p(M(G)) \leq \gamma^{'}(G) $ and $ \gamma^{'}(M(W_n)) = 1 + \lceil \frac{n-3}{3} \rceil $, so $ \gamma_p(M(W_n)) \leq 1 + \lceil \frac{n-3}{3} \rceil $. If $ v_1 $ be a universal vertex of $ W_n $, then it's clear that to be the smallest size of $ PD $-set, $ v_1 $ or one of $ \{ u_{12}, u_{13}, \cdots , u_{1n} \} $ should be in $ PD $-set. Now if $ v_1 $ be in $ PD $-set, then propagation can not occur from any of $ \{ u_{12}, u_{13}, \cdots , u_{1n} \} $, because for example $ u_{12} $ is adjacent to $ \{ v_2, u_{2n}, u_{23} \} $ and this case can not occur that two of these three vertices are monitored and for monitore the remaining vertex, it is needed that propagation occur from $ u_{12} $. So $ V(M(W_n)) \setminus \{ v_1, u_{12}, u_{13}, \cdots , u_{1n} \} $ induce a $ M(C_{n-1}) $ and we have $ \gamma_p(M(W_n)) = 1 + \lceil \frac{n-1}{3} \rceil $ but $ 1 + \lceil \frac{n-3}{3} \rceil \leq 1 + \lceil \frac{n-1}{3} \rceil $ and it is a contradiction. Now, suppose that one of $  \{ u_{12}, u_{13}, \cdots , u_{1n} \} $ be in $ PD $-set. Without losing generality, let $ u_{12} $ be in $ PD $-set. Thus $ \{ u_{13}, \cdots , u_{1n} \} \cup \{ v_2, u_{2n}, u_{23} \} $ are monitored in d.s and like previous argument, propagation can not occur from any of $  \{ u_{13}, \cdots , u_{1n} \}  $. In other hand, $ V(M(W_n)) \setminus \{ v_1, v_2, u_{2n}, u_{23}, u_{12}, u_{13}, \cdots , u_{1n} \} $ induce a $ M(P_{n-2}) $ and we know $ \gamma_p(M(P_{n-2}))=\lceil \frac{n-3}{3} \rceil $, so $ \gamma_p(M(W_n)) = 1 + \lceil \frac{n-3}{3} \rceil $.
\end{proof}

\begin{theorem}[\cite{J87}]
$ \gamma^{'}(M(K_{h,n-h})) = min \{ h, n-h \} $.
\end{theorem}

\begin{theorem}
$ \gamma_p(M(K_{h,n-h})) = min \{ h, n-h \} $.
\end{theorem}
\begin{proof}
Let $ G= K_{h,n-h} $ and $ V(G)=X \cup Y $ such that $ \mid X \mid \leq \mid Y \mid $. If $ S=X $, then $ S $ is a $ PD $-set for $ M(G) $. So $ \gamma_p(M(G)) \leq \mid X \mid $. We Know $ E(G)=\{ u_{ij} \mid v_i ~ and~ v_j~ are~ adjacent \} $, in $ M(G) $, induce $ K_h \Box K_{n-h} $(cartesian product of $ K_h $ and $ K_{n-h} $) and in \cite{soko14} proved that $ \gamma_p(K_h \Box K_{n-h}) =  min \{ h, n-h \} -1 $. Given the position of $ K_h \Box K_{n-h} $ in $ M(K_{h,n-h}) $, it is clear that $ \gamma_p(M(K_{h,n-h})) \geq min \{ h, n-h \} -1$. Without losing generality, let $ min \{ h, n-h \}=h  $. So if $ \gamma_p $-set of $ M(G) $ be $ h-1 $ vertices of $ X $ or $ h-1 $ vertices of $ Y $, then it is clear that propagation can not occur. Also if $ V(K_h)= \{ v_{i_1}, \cdots, v_{i_h} \} $, then $ \{ v_1, \cdots, v_{h-1} \} $ is a $ \gamma_p $-set for $ K_h \Box K_{n-h} $ but this set can not be a $ \gamma_p $-set for $ M(G) $ because propagation can not continue. So $ \gamma_p(M(K_{h,n-h})) \neq h-1 $ and therefore $ \gamma_p(M(K_{h,n-h})) = min \{ h, n-h \} $.
\end{proof}

At the end of this section, we present the following conjecture as an open problem.

\begin{conj}
For any graph $ G $, $ \gamma_p(M(G)) = \gamma^{'}(G). $ 
\end{conj}

%\begin{theorem}
%$ \gamma_p(M(K_n))= $
%\end{theorem}
%\begin{proof}

%\end{proof}

%\begin{theorem}
%$ \gamma_p(M(W_n))= $
%\end{theorem}
%\begin{proof}

%\end{proof}

%\begin{theorem}
%$ \gamma_p(M(K_{h,n-h}))=  $ , for $ h \geq 2, n \geq 4 $.
%\end{theorem}
%\begin{proof}

%\end{proof}

%%%%%%%%%%%%%%%%%%%%%%%%%%%%%%%%%%%%%%%%%%%%%%%%%%%%%%%%%%%%%%%%%%%%%%%%%%%%%%%%%%%%%%
%%%%%%%%%%%%%%%%%%%%%%%%%%%%%%%%%%%%%%%%%%%%%%%%%%%%%%%%%%%%%%%%%%%%%%%%%%%%%%%%%%%%%%%%%%%%%%%%%%
%%%%%%%%%%%%%%%%%%%%%%%%%%%%%%%%%%%%%%%%%%%%%%%%%%%%%%%%%%%%%%%%%%%%%%%%%%%%%%%%%%%%%%%%%%%%%%%%%%

\textbf{Acknowledgments.} The authors would like to thank the referee for his/her careful reading and valuable comments which improved the quality of the manuscript.\\

%%%%%%%%%%%%%%%%%%%%%%%%%%%%%%%%%%%%%%%%%%%%%%%%%%%%%%%%%%%%%%%%%%%%%%%%%%%%%%%%%%%%%%%%%%%%%%%%%%

%%%%%%%%%%%%%%%%%%%%%%%%%%%%%%%%%%%%%%%%%%%%%%%%%%%%%%%%%%%%%%%%%%%%%%%%%%%%%%%%%%%%%%%%%%%%%%%%%%
%%%%%%%%%%%%%%%%%%%%%%%%%%%%%%%%%%%%%%%%%%%%%%%%%%%%%%%%%%%%%%%%%%%%%%%%%%%%%%%%%%%%%%%%%%%%%%%%%%
%%%%%%%%%%%%%%%%%%%%%%%%%%%%%%%%%%%%%%%%%%%%%%%%%%%%%%%%%%%%%%%%%%%%%%%%%%%%%%%%%%%%%%%%%%%%%%%%%%
%%%%%%%%%%%%%%%%%%%%%%%%%%%%%%%%%%%%%%%%%%%%%%%%%%%%%%%%%%%%%%%%%%%%%%%%%%%%%%%%%%%%%%%%%%%%%%%%%%

\end{document}